\pdfoutput=1
\documentclass{article}
\usepackage[utf8]{inputenc}
\usepackage{amsfonts}
\usepackage{amsmath}
\usepackage{mathtools}
\usepackage{amsthm}
\usepackage{amssymb}
\usepackage{mathrsfs}
\usepackage{mathtools}
\usepackage{tikz-cd}
\theoremstyle{definition}
\newtheorem{theorem}{Theorem}[section]
\newtheorem{lemma}[theorem]{Lemma}

\newtheorem{definition}[theorem]{Definition}
\newtheorem{corollary}[theorem]{Corollary}
\newtheorem{proposition}[theorem]{Proposition}
\DeclareMathOperator*{\ulalim}{\underrightarrow\lim}
\newcommand{\HX}{\mathrm{HX}}
\newcommand{\CX}{\mathrm{CX}}
\newcommand{\directlimit}{\ulalim_{i}}
\newcommand{\cech}{\v{C}ech }

\title{Coarse Separation of Coarse $PD(n)$ spaces}
\date{\vspace{-5ex}}
\author{Harsh Patil}
\begin{document}

\maketitle
\begin{abstract}
    We show that if a subspace $A$ of a coarse $PD(n)$ metric space $X$  coarsely separates it then it must have asymptotic dimension at least $n-1$. 
\end{abstract}
\section{Introduction}
A subspace $A$ of a topological space $X$ is said to separate $X$ if the complement $X-A$ consists of more than one connected component.  The classical Jordan curve theorem states that every simple closed curve separates the Euclidean plane and the complement consists of two connected components. 
The analogous notion in the coarse category of metric spaces is that of coarse separation. Informally, a subspace $A$ of a metric space is said to coarsely separates $X$ if there exist two disjoint sets $B_{1}$ and $B_{2}$ such that neither $B_{1}$ nor $B_{2}$ are contained in a finite neighborhood of $A$ and for every $R>0$ there exists a $R'>0$ such that any two points $x_{1}\in B_{1}$ and $x_{2}\in B_{2}$ that lie within a distance of  $R$ from each other are always contained in the $R'$ neighborhood of $A$. 
Recently, the authors of \cite{Bensaid2024CoarseSA} proved various necessary conditions for separating subsets of a variety of spaces including symmetric spaces of non-compact type, horocyclic product of trees, and Bourdon's hyperbolic buildings.  
 In his recent preprint  \cite{tselekidis2024s}, 
Tselekidis makes the following conjecture:
\begin{center}
    $\mathbb{R}^n$ cannot be separated by a subspace of asymptotic dimension  strictly less than $n-1$. 
\end{center}

We give a solution to this conjecture and prove it for the larger collection of Coarse $PD(n)$ spaces:
\begin{theorem}\label{main}
     Let $X$ be a coarse $PD(n)$ space and let $A$ be a subspace of $X$ such that $A$ separates $X$. Then, the asymptotic dimension of $A$ must be at least $n-1$.  
\end{theorem}
Coarse $PD(n)$ spaces were introduced first by Kapovich and Kleiner \cite{KapovichKleiner} as  metric simplicial complexes whose large scale geometric properties are similar to that of a universal cover of a compact manifold. Banerjee and Okun \cite{banerjee2023coarse} recast it in the language of coarse cohomology and and gave a more general definition in the coarse category of metric spaces.

The main tool used in the proof of Theorem \ref{main} is a coarse analogue of Alexander Duality proved by Okun and Banerjee in \cite{banerjee2023coarse}. 

We obtain the following corollary as an immediate consequence
\begin{corollary}
    Let $G$ be a $PD(n)$ group such that $G$ splits along a subgroup $C$. Then, $C$ must have asymptotic dimension at least $n-1$.  
\end{corollary}

\section{Preliminaries}
\subsection{Coarse Separation}
We recall what it means for a subspace to coarsely separate a metric space.
\begin{definition}[Coarse Separation]
Let $X$ be a metric space and $A\subset X$. A subset $C\subset X$ is a \textit{coarse complementary component} of $A$ if for any $R>0$ there exists a $R'>0$ such that $N_{R}(C)\cap N_{R}(X-C)\subset N_{R'}(A)$. 
$C$ is a \textit{shallow} coarse complementary component of $A$ if $C\subset N_{r}(A)$ for some $r>0$. Otherwise it is \textit{deep}. We say that $C$ gives rise to a \textit{deep separation} with respect to $A$ if  both $C$ and $X-C$ are deep coarse complementary components of $X-A$.  Let $A\subset X$. We will say that $A$ \textit{coarsely separates} $X$ if there exists a deep separation of $X$ with respect to $A$. 
\end{definition}

\subsection{Coarse Homology}

Let $X$ be a metric space and let $X^{(n)}$ denote the $n$-fold Cartesian product $X$. An $n$-chain is a formal sum $\alpha=\sum\limits_{\sigma\in X^{(n+1)}} a_{\sigma}\sigma$,  $a_{\sigma}\in \mathbb{Z}$. The \textit{support} of $\alpha$ is defined as the set of all simplices $\sigma$ whose coefficient is non-zero $supp(\alpha)=\{\sigma|a_{\sigma} \neq 0\}$. 
Let $\CX_{n}(X)$ denote the set of  $n$-chains that satisfy the following two properties: 
\begin{enumerate}
    \item For any bounded subset $B$ of $X$, there are only finitely many simplices $\sigma$ such that $a_{\sigma}\neq 0$ and $\sigma$ has all of its vertices in $B$. 
    \item There exists a constant $C>0$ such that  for any $\sigma=(x_{0},x_{1},\dots,x_{n})$ we have,
    $ \textrm{sup }\{ d(x_{i},x_{j})\}_{i,j}<C$
\end{enumerate}
There is a natural boundary map $\partial:\CX_{n}(X)\rightarrow \CX_{n-1}(X)$  defined on $X^{(n+1)}$ by,
$$\partial(x_{0},\dots,x_{n}):= \sum_{i=0}^{n}(-1)^{i}(x_{0},\dots ,\hat{x}_{i},\dots,x_{n}) $$
and extended linearly on all of $\CX_{\ast}(X)$.
The homology of the chain complex $(\CX_{\ast}(X),\partial)$ at the $n$-th
 term is call the \textit{coarse homology} of $X$ and is denoted by $\HX_{n}(X)$.  
 \newline

We do not work directly with the coarse cohomology of the complement in this paper. The only facts about coarse cohomology of the complement that we will use are Proposition \ref{sep} and Theorem \ref{banerjeeokun} stated below. 
For the original definition of coarse cohomology of the complement we refer the reader to \cite{banerjee2023coarse}. 
\begin{proposition}\cite[Lemma 5.6]{banerjee2023coarse}\label{sep}
     $\HX^{1}(X-A)$ is non-trivial if $A$ coarsely separates $X$.
\end{proposition}    
\subsection{coarse $PD(n)$ spaces}
We do not work directly with the definition of coarse $PD(n)$ spaces. We refrain from giving the definition a of coarse $PD(n)$ space. We refer the interested reader to \cite{banerjee2023coarse}. The only fact about coarse $PD(n)$ space used is the theorem by Banerjee-Okun:  
\begin{theorem}\cite[Theorem 1.1]{banerjee2023coarse}\label{banerjeeokun}
    If $X$ is a coarse $PD(n)$ space then for any $A\subset X$, then
    $$\HX^{k}(X-A)\cong \HX_{n-k}(A).$$
\end{theorem}

\subsection{Direct system of groups and chain complexes}
We recall the definition of a direct system of groups and its direct limit. 
\begin{definition}[Direct system of groups]
Let $\{A_{i}\}_{i\in \mathbb{N}}$ denote a sequence of groups and let $\{f^{i,j}\}_{i\leq j}$ denote a collection of maps such that 
\begin{enumerate}
        \item $f^{i,j}$ is a homomorphism from $A_{i}$ to $A_{j}$ for all $i\leq j$,
        \item $f^{i,i}$ is the identity map for all $i$, and
        \item for all $i\leq j \leq k $, $f^{i,k}= f^{j,k}\circ f^{i,j}$. 
    \end{enumerate}
Then, the pair $\langle A_{i}, f^{i,j}\rangle$ is said to be a direct system of groups. 
\end{definition}

\begin{definition}[Direct limit of groups]
    Let $\langle A_{i}, f^{i,j} \rangle$ denote a direct system of groups. Then the direct limit of $\langle A_{i}, f^{i,j} \rangle$, denoted by $\directlimit A_{i} = (A, (\phi_{i})_{i})$, is a group $A$ together with a collection of homomorphisms $\phi_{i} : A_{i} \to A$ such that:
    \begin{enumerate}
        \item $\phi_{j} = \phi_{i} \circ f^{i,j}$ for all $i \leq j$,
        \item For any other group $B$ with homomorphisms $\psi_{i} : A_{i} \to B$ satisfying $\psi_{j} = \psi_{i} \circ f^{i,j}$ for all $i \leq j$, there exists a unique homomorphism $u : A \to B$ such that $\psi_{i} = u \circ \phi_{i}$ for all $i$.
    \end{enumerate}
    The  homomorphisms $\{\phi_{i}\}$ are called the \textit{canonical homomorphisms} associated to the direct limit.  
\end{definition}

A direct limit always exists whenever the direct system consists of abelian groups and it is unique up to isomorphism. We recall the standard construction of the direct limit of a directed system of abelian groups. 
\par
 Let $\langle A_{i},f^{i,j}\rangle$ denote a direct system of abelian groups.  Let $G=\bigoplus A_{i}$. Let $N$ be the subgroup of $G$ generated by the set $\bigcup_{i}\{x-f^{i,j}(x)|x\in A_{i},j\geq i\}$. Define $A:=G/N$. Let $\phi_{i}:A_{i}\rightarrow A$ denote the map obtained by composing  the canonical inclusion  $ \mathfrak{q}:A_{i}\rightarrow G$ by the quotient map $G\rightarrow G/N$.  Then, the pair $(A,\phi_{i})$ is a direct limit for the system $\langle A_{i},f^{i,j}\rangle$. 

\begin{lemma}\label{trivial}
Let $ \langle A_{i},f^{i,j}\rangle$ be a direct system of abelian groups and let $(A,\phi_{i})$ denote its direct limit. Let $g_{1},\dots,g_{k}$ be finitely many elements of  $\bigoplus_{i}A_{i}$ such that, for each $1\leq i\leq k$, $g_{i}$ belongs to $A_{m_{i}}$ for some $m_{i}$. 
    Let $g=g_{1}+g_{2}+\dots+g_{k}\in \bigoplus_{i}A_{i}$.
    $g$ maps to the trivial element under the quotient $G\rightarrow A$ if and only if there exists $M$ such that $M>m_{i}$ for all $1\leq i\leq k$ and $\sum_{i=1}^{k} f^{m_i,M}(g_{i})=0$. 
\end{lemma}

\begin{proof}
The 'if' part of the statement is straightforward. Indeed, if such a $M$ exists then $\mathfrak{q}(g)=(g_{1}-f^{m_1,N}(g_{1}))+\dots + (g_{k}-f^{m_k,N}(g_{k}))\in N$. 
\par
For the converse, we show that the statement holds for all the  elements in $N$.
An arbitrary element $g\in N$ is a sum of finitely many elements of the form $x-f^{i,j}(x)$,
$g=(x_{1}-f^{i_{1},j_{1}}(x_{1}))+\dots + (x_{t}-f^{i_{t},j_{t}}(x_{t}))$.  Let $M>i_{p},j_{p}$ for all $1\leq p\leq t$. 
Then,
\begin{gather*}
    \sum f^{i_{1},M}(x_{1})-f^{j_{1},M}(f^{i_{1},j_{1}}(x_{1}))+\dots + f^{i_{t},M}(x_{t})-f^{j_{t},M}(f^{i_{t},j_{t}}(x_{t}))\\
=\sum f^{i_{1},M}(x_{1})-f^{i_{1},M}(x_{1})+\dots + f^{i_{t},M}(x_{t})-f^{i_{t},M}(x_{t})\\
=0.
\end{gather*}

\end{proof}
We also recall the definition of a direct system of chain complexes and that of its direct limit.  
\begin{definition}[Direct system of chain complexes]
    Let $(A^{i}_{\ast})_{i}$ be a sequence of chain complexes and let $(f^{i,j})_{i\leq j}$ be a collection of maps such that 
    \begin{enumerate}
        \item $f^{i,j}$ is a chain map from $A_{\ast}^{i}$ to $A_{\ast}^{j}$ for all $i\leq j$,
        \item $f^{i,i}$ is the identity map for all $i$, and
        \item for all $i\leq j \leq k $, $f^{i,k}= f^{j,k}\circ f^{i,j}$. 
    \end{enumerate}   
    Then, the sequence $(A^{i}_\ast)_{i}$ together with $\{f^{i,j}\}$ is said to form a direct system of chain complexes. It is denote as $\langle A ^{i}_{\ast},f^{i,j}\rangle$.
\end{definition}
Given a chain map $f:A_{\ast}\rightarrow B_{\ast}$, let $f^{k}$ denote its restriction to the $k$-th chain group $f^{k}:A_{k}\rightarrow B_{k}$. Given a direct system of chain complexes $\langle(A^{i}_{\ast}),f^{i,j}\rangle$, for each $k$, taking the $k$-th chain groups $\{A^{i}_{k}\}_{i}$ together with $f^{k}_{i,j}$ forms a direct system $\langle (A^{i}_{k})_{i},f_{k}^{i,j}\rangle$ of abelian groups. Let $A_{k}:=\directlimit A^{i}_{k}$. Let $\partial:A_{k}\rightarrow A_{k-1}$ denote the natural boundary map induced by the boundary operators $\partial^{i}: A^{i}_{k}\rightarrow A^{i}_{k-1}$ on the individual $A_{i}$s. A more precise description of $\partial$ is as follows: let $\phi_{i}:A^{i}_{k-1}\rightarrow A_{k-1}$ denote canonical homomorphisms from $A^{i}_{k}$ to $A_{k}$. Then, $m^{i}:=\partial^{i}\circ \phi^{i}: A^{i}_{k}\rightarrow A_{k-1}$ yields a family of maps that satisfy $m^{i}=m^{j}\circ f^{i,j}$. There is a unique morphism $ \partial:A_{k}\rightarrow A_{k-1}$ such that $\partial\circ \psi_{i}=\phi_{i}$ .  
In this way, we obtain a chain complex $(A_{\ast},\partial)$ which we call to the direct limit of 
 the system $\langle A^{i}_{\ast},f^{i,j}\rangle$. 
 \par 
 Let $H_{k}(A^{i}_{\ast})$ denote the $k$-th homology group of the chain complex $A^{i}_{\ast}$. Let $H_{k}(f^{i,j}):H_{k}(A^{i}_{\ast}) \rightarrow H_{k}(A^{j}_{\ast})$ denote the map induced by $f^{i,j}$ on the corresponding homology groups. The collection of groups $H_{k}(A^{i}_{\ast})$ together with the homomorphisms $H_{k}(f^{i,j})$ forms a direct system of groups. The following theorem states that the direct limit of the system $\langle H_{k}(A^{i}_{k}), H_{k}(f^{i,j})\rangle$ is isomorphic to $H_{k}(A_{\ast})$. 
\begin{theorem}\label{homology_commutes_with_taking_direct_limit}
    Let $\langle(A^{i}_{\ast}),f^{i,j}\rangle$ denote a direct system of chain complexes and let $(A_{\ast},\partial)$ denote its direct limit. Then, 
    $$H_{k}(A_{\ast})\cong \directlimit H_{k}(A^{i}_{\ast}).$$
\end{theorem}

\subsection{Anti-\cech systems}\label{section_anti_cech}
Let $X$ be a metric space. A cover $\mathcal{U}$ of $X$ by uniformly bounded subsets is said to be a \textit{good cover} if it is locally finite i.e. for each $x\in X$ the set  $\{U \in \mathcal{U}|x\in U\}$ is finite.
 A cover $\mathcal{V}$ is said to be a \textit{refinement} of another cover $\mathcal{V}$ if for every element $V\in \mathcal{V}$ is contained in some element $U\in \mathcal{U}$.

\begin{definition}[Anti-\cech system]

A \textit{Lebesgue number} for a cover $\mathcal{U}$ of a metric space is a number $r>0$ such that any set of diameter $r$ is included in some member of the covering. If every $ x\in X$ is contained in at most $k$ elements of $\mathcal{U}$ then we say that the \textit{multiplicity} of $\mathcal{U}$ is at most $k$.
We recall the definition of asymptotic dimension. 
\begin{definition}
   Let $X$ be a metric space. We say that the asymptotic dimension $asdim(X) \leq n$ if the following condition is satisfied:
\begin{center}
   for every $\lambda>0$, there exists a cover $\mathcal{U}$ of $X$ by uniformly bounded sets with Lebesgue number greater than $\lambda$ and multiplicity at most $n+1$. 
\end{center}
We say that $asdim(X)=k$ if the above condition holds for $n=k$ but not for $n=k-1$. 
\end{definition}
    A sequence $\{\mathcal{U}_{i}\}$ of good covers of a metric space  $X$. We will say that they form an Anti-\cech system if there is a sequence of real numbers $R_{n}\rightarrow \infty$ such that for all $n$,
    \begin{enumerate}
        \item Each set $U$ in $\mathcal{U}_{n}$ has diameter less than or equal to $R_{n}$. 
        \item The cover $\mathcal{U}_{n+1}$ has a Lebesgue number greater than $R_{n}$. 
    \end{enumerate}
\end{definition}
The two conditions in the above definition imply that for each $i<j$, $\mathcal{U}_{i}$ is a refinement of $\mathcal{U}_{j}$. 

\begin{corollary}\label{anti-cech}
    Let $X$ be a metric space such that $asdim(X)=k$. Then $X$ admits an anti-\cech system $\{\mathcal{U}_{i}\}$ such that each cover $\mathcal{U}_{i}$ has multiplicity at most $k+1$. 
\end{corollary}
\begin{proof}
We inductively construct a sequence $\{\mathcal{U}_{i}\}$ of covers with the required properties. Let $R_{0}=1$ and let $\mathcal{U}_{1}$ be a uniformly bounded cover of multiplicity at most $k+1$ such that the Lebesgue number of $\mathcal{U}_{1}$ is at least $R_{0}$. Let $R_{1}=sup\{diam(U)|U\in \mathcal{U}_{1}\}$.  Let $\mathcal{U}_{2}$ be a uniformly bounded cover of multiplicity at most $k+1$ and Lebesgue number at least $R_1$. Proceeding inductively one can obtain a sequence of covers $\{\mathcal{U}_{i}\}$ such that each $\mathcal{U}_{i}$ has multiplicity at most $k+1$ and there is a sequence $R_{n}$ of real numbers tending to infinity such that, for each $i$, every set in $\mathcal{U}_{i}$ has diameter less than or equal to $R_{i}$ and the cover $\mathcal{U}_{i+1}$ has a Lebesgue number greater than or equal to $R_{i}$. 
\end{proof}
\begin{definition}[Nerve of a cover]
 Let $X$ be a metric space and let  $\mathcal{U}$ be a cover of $X$ by uniformly bounded sets. The nerve of $\mathcal{U}$ ,denoted by $N(\mathcal{U})$, is the simplicial complex defined as follows:
 \begin{enumerate}
     \item The zero-skeleton $N(\mathcal{U})^{(0)}$ is $\mathcal{U}$. 
     \item An $(n+1)$-tuple $(U_{0},U_{1}, \dots,U_{n})$ spans an $n$-simplex if and only if the intersection $U_{0}\cap  U_{1}\dots\cap U_{n}$ is non-empty. 
 \end{enumerate}
\end{definition}

We also recall the definition of locally finite homology $H^{lf}_{n}(X)$ of a simplicial complex $X$. Let $X$ be a simplicial complex and let $C^{lf}_{n}(X)$ denote the abelian formed by taking all formal sums $\alpha=\sum\limits_{\sigma\in X^{(n)}} a_{\sigma}\sigma$ , $a_{\sigma} \in \mathbb{Z}$ such that every $\beta\in X^{(n-1)}$  the intersection between the star of $\beta$, $\textrm{st}_{X}(\beta)$, and the support of $\alpha$ is finite. Note that this condition is satisfied automatically if $X$ is locally finite. Then, $C_{\ast}^{lf}(X)$ is a chain complex where the boundary map $\partial:C^{lf}_{n}(X)\rightarrow C^{lf}_{n-1}(X)$ is defined as usual by the following expression on individual simplices 
:
$$\partial[x_{0},\dots,x_{n}]:= \sum_{i=0}^{n}(-1)^{i}[x_{0},\dots ,\hat{x}_{i},\dots,x_{n}] $$
and extended linearly on all $C^{lf}_{n}(X)$. The $n$-th homology of $C_{\ast}$ is said to be the \textit{locally finite homology} of $X$ and denoted by $H^{lf}_{n}(X)$. 
\begin{proposition}\label{anti}
Let $X$ be a metric space and 
    let $\{\mathcal{U}_{i}\}$ and $\{\mathcal{V}_{i}\}$ be any two anti-\cech systems for $X$. Let $\{N(\mathcal{U}_{i})\}$ and $\{N(\mathcal{V}_{i})\}$ be the corresponding direct systems of locally finite simplicial complexes. Then, 
    $$ \directlimit H^{lf}_{k}(N(\mathcal{U}_{i}))\cong \directlimit H^{lf}_{k}(N(\mathcal{V}_{i})).$$
\end{proposition}
\begin{proof}
    The proof is divided into two parts. First, we show that the statement holds for any anti-\cech system $\{\mathcal{U}_{i}\}$ and a subsequence $\{\mathcal{U}_{j(i)}\}$ of it. In the second part, we show that given any two anti-\cech systems $\{\mathcal{U}_{i}\}$ and $\{\mathcal{V}_{i}\}$ one can form an anti-\cech system $\{\mathcal{W}_{i}\}$ such that it contains a subsequence of $\{\mathcal{U}_{i}\}$ and a subsequence of $\{\mathcal{V}_{i}\}$. This, in effect, shows that both $\directlimit H^{lf}_{k}(N(\mathcal{U}_{i}))$ and $\directlimit H^{lf}_{k}(N(\mathcal{U}_{i}))$ are isomorphic to $\directlimit H^{lf}_{k}(N(\mathcal{W}_{i}))$ and hence isomorphic to each other.\par
    Let $\{\mathcal{U}_{i}\}$ be an anti-\cech system and let $\{\mathcal{U}_{j(i)}\}$ be a subsequence of it. Then, there is a natural inclusion map $i:\bigoplus_{j} H^{lf}_{k}( N(\mathcal{U}_{j(i)}))\hookrightarrow \bigoplus_{i} H^{lf}_{k}( N(\mathcal{U}_{i}))$. Let $N_{1}$ denote the subgroup $\langle\bigcup_{i}\{ x-p_{j(i)}^{j(k)}(x)|x\in H^{lf}_{k}( N(\mathcal{U}_{j(i)})), i< k \}\rangle$ and let $N_{2}$ denote the subgroup  $\langle\bigcup_{i}\{ x-p_{i}^{k}(x)|x\in H^{lf}_{k}( N(\mathcal{U}_{i})),i< k \}\rangle$. $N_{1}$ is naturally a subgroup of $N_{2}$. Thus, we get a homomorphism  $q:\bigoplus_{j} H^{lf}_{k}( N(\mathcal{U}_{j(i)}))/N_{1}\rightarrow \bigoplus_{i} H^{lf}_{k}( N(\mathcal{U}_{i}))/N_{2}$. We claim that $q$ is an isomorphism.\par  
    \textit{$q$ is surjective :}
    $\bigoplus_{i} H^{lf}_{k}( N(\mathcal{U}_{i}))/N_{2}$ is generated by elements of the form $\bigoplus \alpha_{i}$ such that all components $\alpha_{i}$  of $\alpha$ are trivial except one. 
    Let $\alpha=\bigoplus \alpha_{i}$ be an element such that $\alpha_{k}=x$ and $\alpha_{i}=0$ for $i\neq k$. Let $m$ be an element of the subsequence $j(i)$ such that $m\geq k$. Consider the element $\beta=\bigoplus_{j(i)} \beta_{j(i)}\in \bigoplus_{j(i)} H^{lf}_{k}( N(\mathcal{U}_{j(i)}))/N_{1}$ such that $\beta_{m}=p_{k}^{m}(x)$ and $\beta_{j(i)}=0$ for all $j(i)\neq m$. Then, $q(\beta)=\alpha$.    
    \par   
    \textit{$q$ is injective :}
    Let $g\in ker(q)$. Let $\bigoplus_{i}g_{j(i)}\in \bigoplus_{i} H^{lf}_{k}( N(\mathcal{U}_{j(i)}))$ be a representative of $g$.   By definition, there exist finitely many indices $m_{1},\dots,m_{k}$ such that $ g_{j(i)}\neq 0$ if and only if $j(i)\in \{m_{1}\dots,m_{k}\}$. Since $q(g)=0$, by the 'only if' part of Lemma \ref{trivial}, there exist a $M$ such that $M>j(i)$ for all $j(i)\in \{m_{1}\dots,m_{k}\}$
 and $\sum _{i=1}^{k}p^{M}_{j(i)}(g_{j(i)})=0$. Let $M'$  be an element of the subsequence $j(i)$ such that $M'>M$. Then, $\sum _{i=1}^{k}p^{M'}_{j(i)}(g_{j(i)})=\sum _{i=1}^{k}p^{M'}_{M}\circ p^{M}_{j(i)}(g_{j(i)})=p^{M'}_{M}(\sum _{i=1}^{k} p^{M}_{j(i)}(g_{j(i)}))=0$. Using the 'if' part of Lemma \ref{trivial}, we have that $g=0$.     
    \par
   Given two anti-\cech systems $\{\mathcal{U}_{i}\}$ and $\{\mathcal{V}_{i}\}$ we give a construction of an anti-\cech system $\{\mathcal{W}_{i}\}$ such that $\{\mathcal{W}_{i}\}$ contains a subsequence of $\{\mathcal{U}_{i}\}$ and a subsequence of $\{\mathcal{V}_{i}\}$. 
  Let $\mathcal{W}_{1}=\mathcal{U}_{1}$ and let $\lambda_{1}=\textrm{sup}_{W\in \mathcal{W}_{1}}diam(W)$. For some index $i$, the cover $ \mathcal{V}_{i}$ will have Lebesgue number greater than or equal to $\lambda_{1}$ and let $j$ be the smallest such index. We then define $\mathcal{V}_{j}$ to be the next term in the sequence, $\mathcal{W}_{2}=\mathcal{V}_{j}$. 
  By induction, assume that we have constructed the first $N$ elements $\mathcal{W}_{1},\dots ,\mathcal{W}_{N}$ of the sequence. Let $\lambda_{N}=\textrm{sup}_{W\in \mathcal{W}_{N}}diam(W)$. If $N$ is even we define $\mathcal{W}_{N+1}$ to be $\mathcal{V}_{i}$, where $i$ is the smallest index such that the cover $\mathcal{V}_{i}$ has Lebesgue number greater than $\lambda_{N}$. 
  If $N$ is odd we define $\mathcal{W}_{N+1}$ to be $\mathcal{U}_{i}$, where $i$ is the smallest index such that the cover $\mathcal{U}_{i}$ has Lebesgue number greater than $\lambda_{N}$.
  The sequence $\mathcal{W}_{i}$ alternates between elements of $\{\mathcal{U}_{i}\}$ and $\{\mathcal{V}_{i}\}$. It is an anti-\cech system by construction. This concludes the proof. 
\end{proof}

\section{Proof of Theorem \ref{main}}
In this section, we give a proof of Theorem \ref{main}. We first prove Theorem \ref{direct_limit} which allows us to compute the coarse homology of a space using an Anti-\cech system.  

The following theorem can be found in the PhD thesis of S.M. Hair \cite{SMHairThesis} albeit without a proof. It has also been stated in \cite{Roe} where it is taken to be the definition of Coarse homology rather than a theorem. We provide a proof here for the sake of completeness. 
\begin{theorem}\label{direct_limit}
    Let $X$ be a metric space and let $\{\mathcal{U}_{i}\}$ be a anti-\cech system for $X$. Then there is an isomorphism 
    $$\HX_{k}(X)\cong \directlimit H^{lf}_{k}(N(\mathcal{U}_{i}))$$
\end{theorem}

\begin{proof}
    In view of Proposition \ref{anti}, it suffices to exhibit an isomorphism between $\HX_{k}(X)$ and $\varinjlim H^{lf}_{k}(N(\mathcal{U}_{i}))$ for any one anti-\cech system $\{ \mathcal{U}_{i} \}$.
    The proof consists of two steps. In the first step, we express the chain complex $\CX_{\ast}(X)$ as the direct limit of a direct system $\langle\CX_{\ast}^{\lambda_{i}}(X),c_{i.j}\rangle$ of subcomplexes. In the second step, we show that the direct system $\langle\CX_{\ast}^{\lambda_{i}}(X),c_{i.j}\rangle$ is isomorphic to the direct system $\langle C^{lf}_{\ast}(N(\mathcal{U}_{i})),p^{j}_{i}\rangle$ for a particular Anti-\cech system $\{\mathcal{U}_{i}\}$. \par
    Let $\{ \lambda_{i} \}$ be a monotonically increasing sequence of positive reals such that $\lambda_i \to \infty$.  Let $\CX_{\ast}(X)$ be as defined in Section 2.2, and let $\CX_{\ast}^{\lambda_{i}}(X)$ denote the subcomplex of $\CX_{\ast}(X)$ consisting of chains $\alpha = \sum\limits_{\sigma \in X^{(n+1)}} a_{\sigma} \sigma$ such that:
    \begin{enumerate}
        \item For any bounded subset $B$ of $X$, there are only finitely many simplices $\sigma$ such that $a_{\sigma} \neq 0$ and $\sigma$ has all of its vertices in $B$, and 
        \item For any $\sigma = (x_{0}, x_{1}, \dots, x_{n}) \in \operatorname{supp}(\alpha)$, there exists $y \in X$ such that $y \in B_{\lambda_{i}}(x_{i})$ for all $0 \leq i \leq n$.
    \end{enumerate}

We denote the boundary map of $\CX^{\lambda_{i}}_{\ast}(X)$ by $\partial^{\lambda_{i}}$. It should be clear that $\partial^{\lambda_{i}}$ is simply the restriction of the boundary map $\partial$ of $\CX_{\ast}(X)$ to the subcomplex $\CX^{\lambda_{i}}_{\ast}(X)$. 
    Note that there is a natural inclusion map $c^{k}_{i,j}:\CX_{k}^{\lambda_{i}}(X)\rightarrow \CX_{k}^{\lambda_{j}}(X)$ for all $i<j$ and for all $k$. It is not hard to see that $c^{k}_{i,j}$  commutes with the boundary operator and extends to a chain map $c_{i,j}:\CX_{\ast}^{\lambda_{i}}(X)\rightarrow \CX_{\ast}^{\lambda_{j}}(X)$. This gives rise to 
 a direct system of chain complexes $\langle\CX_{\ast}^{\lambda_{i}}(X), c_{i,j}\rangle$. Additionally, for each $k$, the collection $\langle \CX^{\lambda_{i}}_{k}(X),c^{k}_{i,j}\rangle$ is a direct system of abelian groups and the direct limit of this system is naturally isomorphic to $\CX_{k}(X)$. The direct limit of the system $\langle\CX_{\ast}^{\lambda_{i}}(X), c_{i,j}\rangle$ is the chain complex obtained by taking the abelian groups $\CX_{k}(X)$ together with boundary map obtained by taking the direct limit of the boundary maps $\partial^{\lambda_{i}}$, $\delta:=\varinjlim_{i}(\partial^{\lambda_{i}})$. It can be easily verified that the map $\delta$ coincides with the original boundary map $\partial$. 
 \par
 The chain map $c_{i,j}$ induces a map $H_{k}(c_{i,j})$ on the corresponding homology groups $H_{k}(c_{i,j}):H_{k}(\CX_{\ast}^{\lambda_{i}}(X))\rightarrow H_{k}(\CX_{\ast}^{\lambda_{j}}(X))$. For each $k$, the collection of groups 
$\langle H_{k}(\CX_{\ast}^{\lambda_{i}}(X)),H_{k}(c_{i,j})\rangle$ forms a direct system. By Theorem \ref{homology_commutes_with_taking_direct_limit}, the direct limit $\varinjlim_{i} H_{k}(\CX_{\ast}^{\lambda_{i}}(X))$ is isomorphic to the $k$-th homology group of the complex $\CX_{\ast}(X)$, i.e. $\HX_{k}(X)\cong \varinjlim_{i} H_{k}(\CX_{\ast}^{\lambda_{i}}(X))$. 
  \par 
  Let $\mathcal{U}_{i}$ denote the cover of $X$ consisting of balls of radius $\lambda_{i}$. Then $\{ \mathcal{U}_{i} \}$ forms an anti-\cech  system. Let $C^{lf}_{\ast}(N(\mathcal{U}_{i}))$ denote the chain complex of locally finite chains on $N(\mathcal{U}_{i})$. For all $n$, there is a natural bijection between $n$-simplices in $N(\mathcal{U}_{i})$ and simplices $\sigma \in X^{(n+1)}$, which sends $(x_{0}, x_{1}, \dots, x_{n})$ to the simplex $(B_{\lambda_{i}}(x_{0}), B_{\lambda_{i}}(x_{1}), \dots, B_{\lambda_{i}}(x_{n}))$. This bijection induces an isomorphism
    \[
    \theta^{i}_{n} : \CX_{n}^{\lambda_{i}}(X) \rightarrow C^{lf}_{n}(N(\mathcal{U}_{i}))
    \]
    for all $n$. The isomorphisms $p_{n}$ extend to a chain map
    \[
    \theta^{i} : \CX_{\ast}^{\lambda_{i}}(X) \rightarrow C^{lf}_{\ast}(N(\mathcal{U}_{i})).
    \]
    $p$ is easily seen to be a chain homotopy equivalence by considering the map induced by the inverse of the bijection between $n$-simplices in $N(\mathcal{U}_{i})$ and simplices $\sigma \in X^{(n+1)}$. 
    As a result, the corresponding homology groups are isomorphic:
    \[ H_{k}(\theta^{i}):
    H_{k}(\CX_{\ast}^{\lambda_{i}}(X)) \xrightarrow{\sim} H^{lf}_{k}(N(\mathcal{U}_{i})). 
    \]
   It remains to show that the isomorphisms $H_{k}(p)$ extend to an isomorphism between direct systems $
   \langle H_{k}(\CX_{\ast}^{\lambda_{i}}(X), H_{k}(c_{i,j})\rangle$ and $\langle H^{lf}_{k}(N(\mathcal{U}_{i}),H_{k}(p^{j}_{i})\rangle$. 
   It can be easily verified that the following diagram commutes for all $i\leq j$,
   \[ \begin{tikzcd}
 H_{k}(\CX_{\ast}^{\lambda_{i}}(X))\arrow{r}{ H_{k}(\theta^{i})} \arrow[swap]{d}{H_{k}(c_{i,j})} & H^{lf}_{k}(N(\mathcal{U}_{i})) \arrow{d}{H_{k}(p^{j}_{i})} \\%
H_{k}(\CX_{\ast}^{\lambda_{j}}(X)) \arrow{r}{H_{k}(\theta^{j})}& H^{lf}_{k}(N(\mathcal{U}_{j}))
\end{tikzcd}
\]
   It follows that the corresponding direct limits as isomorphic, $$\HX_{k}(X)\cong\varinjlim_{i} H_{k}(\CX_{\ast}^{\lambda_{i}}(X)) \cong \directlimit H^{lf}_{k}(N(\mathcal{U}_{i})).$$
\end{proof}
\begin{corollary}\label{final}
    Let $X$ be a metric space such that $X$ has finite asymptotic dimension. Then,
    $$asdim(X)\geq \textrm{sup}\{k|\HX_{k}(X)\neq 0 \}.$$
\end{corollary}
\begin{proof}
  Let $asdim(A)=n$. Then, by Corollary \ref{anti-cech}, there exists a anti-\cech system $ \{\mathcal{U}_{i}\}$ for $X$ such that each cover $\mathcal{U}_i$ has multiplicity $n+1$.  Thus, the corresponding nerves $N(\mathcal{U}_{i})$ all have dimension at most $n$. It follows from Theorem \ref{direct_limit} that $\HX_{i}(X)$ vanishes for all $i>n$. 
\end{proof}
Finally, we give a proof of Theorem \ref{main}:
\begin{proof}
    Let $X$ be a coarse $PD(n)$ space. Let $A$ be a subspace of $X$ such that $A$ separates $X$. By Proposition \ref{sep} and Theorem \ref{banerjeeokun}, $\HX_{n-1}(A)\cong\HX^{1}(X-A)\neq 0$. An application of Corollary \ref{final} tells us that the asymptotic dimension of $A$ is at least $n-1$.
\end{proof}
\bibliographystyle{plain}
\bibliography{main}
\end{document}